\documentclass[12pt]{article}

\usepackage{amsmath,amsthm,amsfonts,amssymb}

\newcommand{\Md}[4]{\left(\begin{array}{cc}{#1}&{#2}\\{#3}&{#4}\end{array}\right)}

\newcommand{\R}{\mathbb{R}}
\newcommand{\rar}{\R^2\rtimes_A\R}

\newcommand{\Sol}{{\rm Sol}_3}

\newcommand{\hn}[1]{{\mathbb{H}^{#1}}}
\newcommand{\forma}[1]{\langle #1 \rangle}

\newcommand{\G}{\Gamma}
\newcommand{\wt}{\widetilde}

\newcommand{\grad}{\mbox{\rm grad}}

\newcommand{\abs}[1]{\vert #1 \vert}

\newcommand{\hh}{\mathbb{H}^2}

\usepackage{color}
\usepackage{amsmath}
\usepackage{amsfonts}
\usepackage{amssymb}
\usepackage{graphicx}%
\usepackage{fullpage}
\usepackage{graphicx}
\usepackage{enumerate}
\usepackage{xcolor}
\usepackage{url}

\usepackage{times}
\usepackage[hidelinks]{hyperref}
\usepackage{latexsym}

\providecommand{\U}[1]{\protect\rule{.1in}{.1in}}
\newtheorem{theorem}{Theorem}
\newtheorem{thm}[theorem]{Theorem}

\newtheorem{claim}[theorem]{Claim}

\newtheorem{definition}[theorem]{Definition}

\newtheorem{proposition}[theorem]{Proposition}

\theoremstyle{definition}
\newtheorem{remark}[theorem]{Remark}

\begin{document}

\title{On the non-parabolicity of $\Sol$}
\title{On the non-parabolicity of $\Sol$}
\author{L. Bonorino, G. Nunes,
A. Ramos\footnote{The third, the fourth and the sixth authors were partially supported
by CNPq/Brazil}\, ,
J. Ripoll$\vphantom{o}^*$,
L. Sauer and
M. Telichevesky$\vphantom{o}^*$}

\maketitle

\begin{abstract}
We prove that $\Sol$, the isometry group of the Minkowski plane, 
is non parabolic with respect to any left invariant metric.
\end{abstract}


\

\section{Introduction.}

A complete Riemannian manifold $M$ is called {\em parabolic} if any
(entire) positive
superharmonic function of $M$ is constant; if the contrary
happens then $M$ is called {\em non-parabolic}. Although recently
the more general notion
of $p-$parabolicity (see~\cite{G}) has attracted the attention of the
mathematicians, there are still some interesting open problems related to
parabolicity (and non-parabolicity) of manifolds. 
For instance, Green-Wu~\cite{GW} conjectured that 
if the sectional curvature $K$ of a Hadamard
manifold $M$ satisfies $K\leq-C/r^{2}$, $r\geq r_{0}>0,$ for some
positive constant $C,$ where $r=r(p),$ $p\in M,$ is the Riemannian distance of
$p$ to a given point of $M,$ then there exist bounded, nonconstant 
harmonic functions on $M$, which, in particular, implies non-parabolicity.

Recently, L. Priebe and R. Soares proved that if the Ricci
curvature of $M$ is non negative and decays to $0$ at most exponentially, then
$M$ is parabolic (in fact, that $M$ is $p-$parabolic for any $p>1$, 
see~\cite[Theorem~1.1]{PS}). Also, a theorem of S. T. Yau
\cite{SY} proves that if $M$ has nonnegative Ricci curvature then any
positive \emph{harmonic} function is constant. Since $\mathbb{R}^{3}$ is non
parabolic we cannot weaken the decay condition of Priebe-Soares result to
nonnegative Ricci curvature or, equivalently, replace harmonic by
superharmonic in Yau's result. However, it is not known if this decay condition is sharp.

There are many conditions for proving the parabolicity of a Riemannian
manifold $M,$ closely connected with the behavior of the sectional (or Ricci) 
curvature
$K\ $of $M$. These conditions are difficult to apply when $K$ changes sign on
unbounded domains of $M$ with an uniform variation bounded from below, i.e.,
with $K^{+},\left\vert K^{-}\right\vert \geq k>0,$ where $K^{+}=\max\{K,0\}$
and $K^{-}=\min\{K,0\}$. An interesting class of such manifolds are Lie
groups endowed with left invariant metrics 
(see Theorems 2.4, 2.5 and comments after
Theorem 2.5 in \cite{Milnor1}). 
From what is known, as we may see below,  
the general idea is that such Riemannian manifolds are non-parabolic, but
the fact that there
there are few general conditions to test parabolicity, it 
seems to be necessary an ad hoc study of these manifolds.

In~\cite{I.H}, I. Holopainen proves that the
Heisenberg groups, when endowed with a left invariant metric, are always non parabolic.
Holopainen's paper seems to be one of the few works treating explicitly the
parabolicity problem on an special and well known family of Lie groups with a left
invariant metric (recall that a Lie group endowed with
a left invariant metric is called a {\em metric Lie group}, see~\cite{MP}). 
In the present manuscript, we prove
that the group $\Sol$, the isometry group of the Minkowski plane, is non parabolic with respect to any left invariant metric:

\begin{thm}\label{thmmain}
Let $G$ be a metric Lie group. If $G$ is isomorphic to $\Sol$, then $G$ is non-parabolic.
\end{thm}

When endowed with a special left-invariant metric (namely the metric which contains the largest number of symmetries), the group $\Sol$
is one of the eight {\em Thurston's geometries}:
$$\mathbb{R}^3,\,\mathbb{H}^3,\,\mathbb{S}^3,\,
\mathbb{H}^2\times\mathbb{R},\,\mathbb{S}^2\times\R,\,
\widetilde{SL}(2,\mathbb{R}),\,{\rm Nil}_3,\,\Sol.$$
Here, $\mathbb{R}^n,\,\mathbb{S}^n$ and $\mathbb{H}^n$ are the space forms
of dimension $n$, $\widetilde{SL}(2,\mathbb{R}),$ is the
universal covering of the special linear group of $2\times2$ matrices, and
${\rm Nil}_3$ is the
$3$-dimensional Heisenberg group. With the unique exception of
$\mathbb{S}^2\times\R$ (which is parabolic), all Thurston's geometries are
metric Lie groups. The Heisenberg group is non parabolic by Holopainen work;
$\mathbb{R}^{3},$ $\mathbb{H}^{3},$ $\mathbb{H}^{2}\times\mathbb{R}$ are well
known to be non parabolic, and Theorem~\ref{thmmain} places $\Sol$ as 
yet another
non parabolic Thurston geometry defined by a metric Lie group. 
The parabolicity (or non parabolicity) of the remaining case,
$\widetilde{SL}(2,\mathbb{R)},$ seems to be unknown.

The proof of Theorem \ref{thmmain} uses an elementary direct approach,
carried out in Section~\ref{secproofs}: it consists in
constructing an explicit example of a positive, entire, nonconstant
harmonic function in $\Sol$, by choosing a special one parameter
subgroup $\G$ of isometries of $\Sol$ and studying a certain 
elliptic partial differential equation in the
quotient space $\Sol/\G$ that comes from the Laplacian operator of
$\Sol$. We observe  that Theorem 1, proving the existence of a positive harmonic function in $\Sol$, proves indeed that $\Sol$ does not satisfy the Liouville property, according to J. Kazdan ~\cite{JK}.

\section{The semidirect product representation of $\Sol$.}

Next, we introduce some definitions and state some facts that 
will be used in the 
proof of Theorem~\ref{thmmain}. First, let $A\in M_2(\R)$ be a $2\times 2$ real matrix and let, for
each $z\in \R$, $e^{Az}$ 
be its exponential map, acting on $\R^2$ via left multiplication.
The semidirect product
$\rar$ is the Lie group $(\R^3,*)$, where $*$ is the operation defined by
$$({\bf p}_1,z_1)*({\bf p}_2,z_2) = ({\bf p}_1+e^{Az_1}{\bf p}_2,
z_1+z_2).$$ 
If we denote
\begin{equation}\label{eAz}
e^{Az} = \Md{a_{11}(z)}{a_{12}(z)}{a_{21}(z)}{a_{22}(z)},
\end{equation}
the group operation of $\rar$ in coordinates can be expressed as
$$(x_1,y_1,z_1)*(x_2,y_2,z_2) = (x_1+a_{11}(z_1)x_2+a_{12}(z_1)y_2,
y_1+a_{21}(z_1)x_2+a_{22}(z_1)y_2, z_1+z_2).$$
In order to regard $\rar$ as a metric Lie group,
we next present the
{\em canonical left-invariant metric} of $\rar$
(see Meeks-Pérez~\cite[Section~2.3]{MP} for more details).
\begin{definition}
The {\em canonical left-invariant metric} on $\rar$ is such that
$\{\partial_x,\,\partial_y,\,\partial_z\}$ is an orthonormal 
basis at the origin (0,0,0). In coordinates, the vector fields
$$E_1 = a_{11}(z) \partial_x + a_{21}(z)\partial_y,\quad
E_2 = a_{12}(z) \partial_x + a_{22}(z)\partial_y,\quad
E_3 = \partial_z$$
form an orthogonal frame of left invariant vector fields extending
$\{\partial_x,\partial_y,\partial_z\}$ at the origin.
\end{definition}
Using this notation, it is possible to obtain a one-parameter family of
metric semidirect products $\rar$ which, up to rescaling, 
will serve as models for the
group $\Sol$ endowed with any left-invariant metric.

\begin{proposition}\label{propSol} 
Let $\Sol$ be the group of isometries of the
Minkowski plane. Then, if $g$ is any left-invariant metric on $\Sol$, there
exists $a\geq0$ such that the metric Lie group $(\Sol,g)$ is, 
up to homothety, isomorphic and isometric to $\rar$ (endowed with its 
canonical left-invariant metric), where
\begin{equation}\label{Asol}
A = \Md{1}{a}{0}{-1}.
\end{equation}
\end{proposition}

\begin{remark}
The proof of Proposition~\ref{propSol} is carried
out in Section~2.7 of~\cite{MP} with a minor distinction. 
There, the authors prove that, when
endowed with a left-invariant metric,
$\Sol$ is isomorphic and isometric to a semidirect product
$\R^2\rtimes_{B}\R$ where 
$$B = \Md{0}{c_1}{1/c_1}{0},$$
for some $c_1\geq1$. The stated representation with the matrix
$A$ of~\eqref{Asol} follows after observing that $A$ and $B$ as above are
congruent (which makes $\rar$ and $\mathbb{R}^2\rtimes_B\R$ both isomorphic 
and isometric) when $a = (c_1^2-1)/c_1$.
\end{remark}

Henceforward, $\Sol$ will denote the metric Lie group modelled by $\rar$,
where $a\geq0$ and $A$ is given by~\eqref{Asol}.

\section{The proof of Theorem~\ref{thmmain}.}\label{secproofs}

In this section, we show that $\Sol$ is non-parabolic by exhibiting
an explicit example of an entire, positive, nonconstant
harmonic function. 

First, we prove Theorem~\ref{thmmain} in the case when $a = 0$, that is, when
$\Sol$ is modelled by $\rar$ where $A$ is
\begin{equation}\label{1-1}
A=\Md{1}{0}{0}{-1}.
\end{equation} 
We note that this is the mostly well known model for $\Sol$, which
makes it one of the eight Thurston's geometries. Explicitly, this model
is $(\R^3,ds^2)$, where
\begin{equation*}
ds^2  =  e^{-2z}dx^2+e^{2z}dy^2+dz^2.
\end{equation*}
We note that the frame $\{E_1 = e^z\partial_x,\, E_2 = e^{-z}\partial_y,\,
E_3 = \partial_z\}$ is composed by left invariant vector
fields which are unitary and everywhere orthogonal.

Our next construction is to produce an entire, positive, nonconstant harmonic function $u\colon \Sol\to \R$, which will not depend on
the $x$ variable. Recall that if $u$ is smooth, being harmonic
is equivalent to $\Delta u = 0$, where 
$\Delta$ denotes the Laplacian operator on $\Sol$, which can be 
written in coordinates as
\begin{equation}\label{lapla}
\Delta u= e^{2z}u_{xx}+e^{-2z}u_{yy}+u_{zz}.
\end{equation}
We find a function as described above after 
constructing a Riemannian submersion
$P\colon \Sol\to \hn2$ and defining
$u = w\circ P$, where $w\colon \hn2\to \R$ is a suitable function. 

Consider the right-invariant (and thus Killing) vector field $\partial_x$,
whose flux acts on $\Sol$ via the 1-parameter group of isometries 
$$\Gamma = \{(x,y,z)\mapsto (x+t,y,z)\}_{t\in\R}.$$
Let $M = \{(0,y,z)\in \rar\mid y,z\in\R\}$, endowed with the induced
ambient metric. The next claim presents two key properties
for our construction.

\begin{claim}\label{newlemma}
$M$ is isometric to the hyperbolic plane $\hn2$ and
the map $\pi\colon \Sol \to M$ defined by
$\pi(x,y,z) = (0,y,z)$ is a Riemannian submersion. 
\end{claim}
\begin{proof}
To see that $M$ is isometric to $\hn2$, just note that the ambient metric
restricted to $M$ is simply $e^{2z}dy^2+dz^2$ and
the map $$(0,y,z)\in M\mapsto (y,e^{-z})\in \R^2_+ = \{(x,y)\in \R^2\mid y>0\}$$
is an isometry between $M$
and the half-space model for $\hn2$, $$(\R^2_+,\frac{dx^2+dy^2}{y^2}).$$
The fact that $\pi$ is a Riemannian submersion follows from observing that
the fibers are horizontal lines $\{(t,y_0,z_0)\mid t\in \R\}$, 
so ${\rm ker}(d\pi)$ is generated by $\partial_x$ and
${\rm ker}(d\pi)^\perp$ is generated by $\{\partial_y,\partial_z\}$.
The fact that $\pi$ leaves the third coordinate unchanged then makes
the restriction of $d\pi$ to ${\rm ker}(d\pi)^\perp$ an isometry.
\end{proof}

If $v\colon \hn2 \to \R$ is a smooth function, we let
$\wt{v} = v\circ P\colon \Sol \to \R$ denote the lift of $v$ to $\Sol$ by $P$. 
For a given $p\in\Sol$, let $\Gamma(p)$ denote its orbit with respect
to the action of $\G$ and
let $\overrightarrow{H}$ denote the mean curvature vector of $\G(p)$,
which is always orthogonal to the fibers of $\pi$. 
Since $\overrightarrow{H}$ is $\G$-invariant, its projection
$d\pi \overrightarrow{H}$ defines a vector field in $M$,
which we denote by $J$, so
$J = dP\overrightarrow{H}$ 
is a vector field in $\hn2$. Under these conditions, the
same proof of~\cite[Proposition~3]{RT} applies to
find that $\Delta\wt{v}= 0$
if and only if $v$ satisfies
\begin{equation}\label{eqimportant}
\Delta v - \forma{\grad(v),J}= 0,
\end{equation}
where $\Delta$ and $\grad$
respectively denote the Laplacian and the gradient 
in $\hn2$ and $\forma{\,,\,}$ is the hyperbolic metric.
Next, we find an explicit expression
for $J$ in coordinates, and we refer to~\cite[Equation~(2.11)]{MP} for
the Riemannian connection of the semidirect product model of $\Sol$.

\begin{claim}
For any $p = (0,y,z)\in M$, $J = \partial_z$. In particular, 
$J$ is the unitary tangent field to a family of geodesics of
the hyperbolic metric on $M$, all issuing from 
the same point $p^*$ at infinity (as in Figure~\ref{figOrbits}).
\end{claim}
\begin{figure}
\centering
\includegraphics[width=0.6\textwidth]{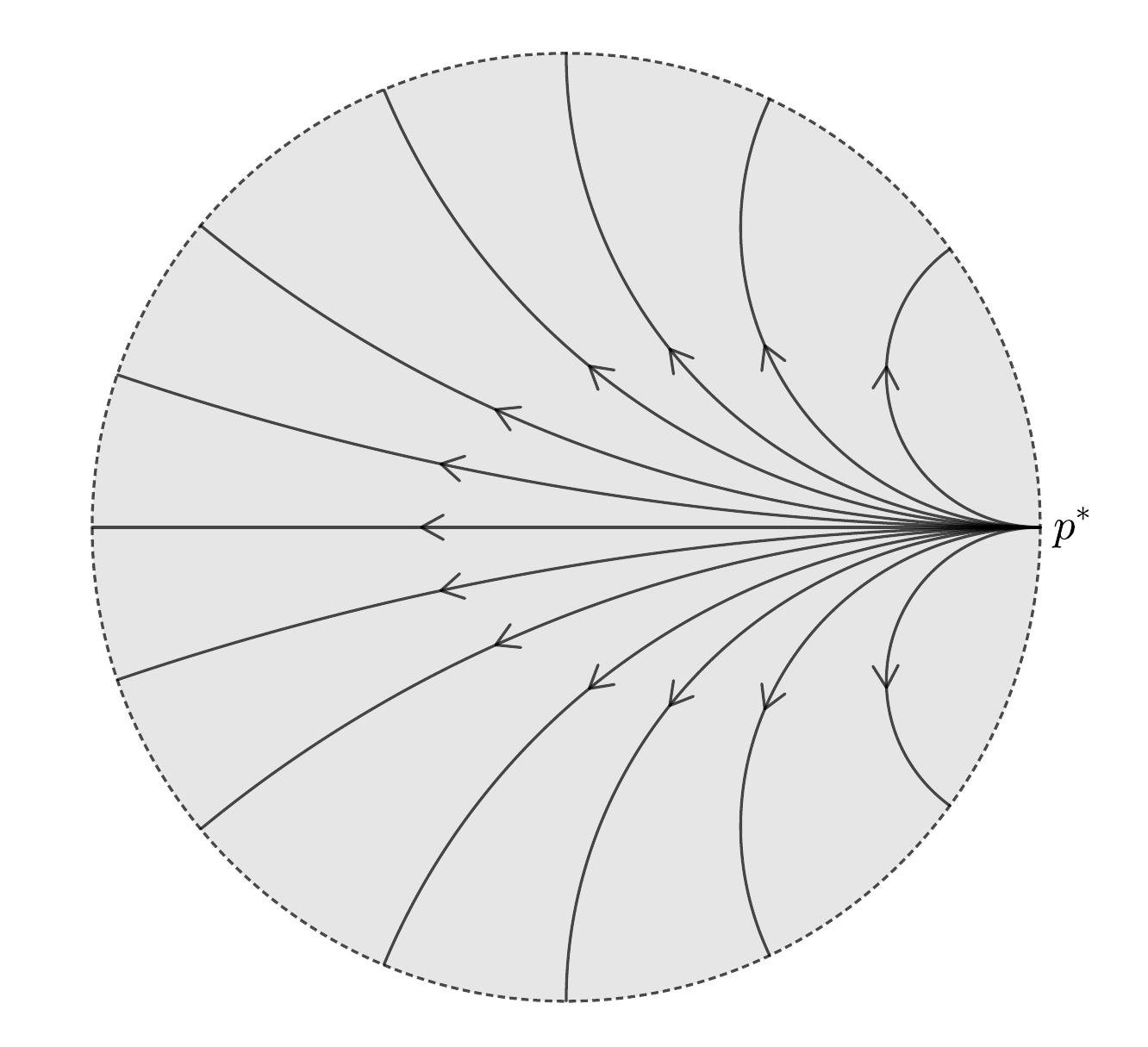}
\caption{\label{figOrbits}Depicted in the disk model for $\hn2$, 
the orbits of $J$ are geodesics issuing
from the point at infinity $p^*$. }
\end{figure}
\begin{proof}
First, note that $E_1 = e^z\partial_x$ 
is a unitary vector
field tangent to the orbits $\Gamma(p)$, which are all orthogonal to 
the vertical planes $\mathcal{P}_c = \{(c,y,z)\in \rar \mid y,z\in\R\}$.
Therefore,
if $(\cdot)^\bot$ denotes the orthogonal projection onto $T\mathcal{P}_c$,
$$
\overrightarrow{H} =
\left(
\nabla_{E_1}E_1
\right)^{\bot}
=\left(E_3\right)^{\bot} = E_3 = \partial_z.
$$
To finish the proof of the claim, just note that
$d\pi \partial_z = \partial_z|_{\hn2}$, so the orbits of the flux of $J$
are vertical lines $\{(0,y_0,t)\mid t\in \R\}$, 
which are geodesics as described before.
\end{proof}

For a function $u\colon \hn2\to \R$, let 
$L(u) =\Delta u -  \forma{\grad u, J}$.
The next step in the proof of Theorem~\ref{thmmain} is to find a smooth solution $w\in C^0(\mathbb{H}^2)$ to the linear
partial differential equation
\begin{equation}\label{eqoperator}
L(u) =0.
\end{equation}

Notice that $J$ is orthogonal to the horocycles having $p^*$ as point 
at infinty. More precisely, chosing any horocycle $\mathcal{H}$ of this 
family and denoting by $s$ the signed distance function to $\mathcal{H}$ 
(pointing towards the concave side defined by $\mathcal{H}$ in $\hn2$),
then,
$$J=\nabla s.$$
Furthermore, it is well-known that $\Delta s \equiv 1$ in $\hh$.

Write $w:\hh\to\mathbb{R}$ 
$$w(p)=e^{s(p)/2}$$ 
and observe that 
$$\nabla w = \frac{1}{2} w \nabla s.$$
Since $\abs{\nabla s} = 1 = \Delta s$,
\begin{equation}\label{laplaw}
\Delta w = \frac{1}{4} w\vert\nabla s\vert^2+ \frac{1}{2}w \Delta s 
= \frac{3}{4} w.
\end{equation}

To finish the construction, choose any point $o\in \hh$ and let 
$r\colon\hh \to \mathbb{R}$ be the distance function to $o$, i.e.,
$r(p)=d(p,o)$ in $\hh$. Let $v\colon \hh \to \mathbb{R}$ be the radial 
eigenfunction of $\hh$ related to its first eigenvalue $\lambda_1=1/4$
satisfying $v(o)=1$. More specifically, $v$ satisfies
\begin{equation}\label{laplav}
\Delta v+\frac{1}{4} v = 0.
\end{equation}
It is well known that 
$v$ is a non-constant, positive function.
Finally, define $u\colon \hh \to \mathbb{R}$ by $u=vw$.
Therefore, by~\eqref{laplaw} and~\eqref{laplav},
\begin{eqnarray}
\Delta u &=& v\Delta w + w\Delta v + 2\langle \nabla v, \nabla w \rangle \nonumber\\&=& \frac{3}{4} vw - \frac{1}{4} vw + w \langle \nabla v, \nabla s\rangle\nonumber \\ &=& \frac{1}{2} vw + wJ(v).\label{Deltau}
\end{eqnarray}

On the other hand, since $J(w) = \frac12 w$,
\begin{equation}\label{Ju}
J(u) = vJ(w) + wJ(v) = \frac{1}{2} vw + wJ(v).
\end{equation}

Combining equations \eqref{Deltau} and \eqref{Ju}
 we obtain that \[\Delta u - J(u) = 0,\] therefore $u$ is a positive non-constant solution  to~\eqref{eqoperator}. Thus,
as already explained, this implies that 
the function $\wt{u}\colon \Sol \to \R$ defined as
\begin{equation}\label{defW}
\wt{u}(x,y,z)=u(\pi(x,y,z))= u(0,y,z)
\end{equation}
is a positive, nonconstant harmonic 
function in $\Sol$, proving that 
$\Sol$ is non parabolic when endowed with the canonical metric defined
by the matrix $A$ as in~\eqref{1-1}.

To prove Theorem~\ref{thmmain} in the general case, fix $a>0$.
Let
$$A = \Md{1}{a}{0}{-1},$$
and consider $\Sol$ as the semidirect product $\rar$ endowed with its 
canonical left-invariant metric.
In this case, 
$$e^{Az}=\Md{e^z}{p(z)}{0}{e^{-z}},$$ where $p(z)$ is a polynomial on $z$.
The metric on $\Sol$ now is given by
$$ds^2 = e^{-2z}dx^2+\left((p(-z))^2+e^{2z}\right)dy^2+dz^2
+p(-z)e^zdxdy,$$
and the Laplacian operator of $\Sol$ is
\begin{equation}\label{laplSola}
\Delta u= (e^{2z}+p(z)^2)u_{xx}+2p(z)e^{-z}u_{xy}+e^{-2z}u_{yy}+u_{zz}.
\end{equation}
Since the function $\wt{u}$ defined in~\eqref{defW} 
does not depend on the variable $x$, it also satisfies (\ref{laplSola}), which
concludes the proof of the theorem.

\center{Instituto de Matemática e Estatística, Universidade Federal do Rio Grande do Sul, Brazil.\\
Email address:  bonorino@mat.ufrgs.br}
\center{
Instituto de Física e Matemática, Universidade Federal de Pelotas,
Brazil.\\
Email address: giovanni.nunes@ufpel.edu.br}

\center{Instituto de Matemática e Estatística, Universidade Federal do Rio Grande do Sul, Brazil.\\
Email address:  alvaro.ramos@ufrgs.br}

\center{Instituto de Matemática e Estatística, Universidade Federal do Rio Grande do Sul,
 Brazil.\\
 Email address:  jaime.ripoll@ufrgs.br}

\center{Instituto de Física e Matemática, Universidade Federal de Pelotas,
 Brazil.\\
 Email address: lisandra.sauer@ufpel.edu.br}

\center{Instituto de Matemática e Estatística, Universidade Federal do Rio Grande do Sul, Brazil.\\
Email address:  miriam.telichevesky@ufrgs.br}

\end{document}